\documentclass[12pt]{article} 
\usepackage{amsfonts,amsmath,latexsym,amssymb,mathrsfs,amsthm,comment}
\usepackage{slashbox}
\usepackage{caption}

\let\OLDthebibliography\thebibliography
\renewcommand\thebibliography[1]{
  \OLDthebibliography{#1}
  \setlength{\parskip}{3pt}
  \setlength{\itemsep}{0pt plus 0.3ex}
}

\def\numberlikeadb{\global\def\theequation{\thesection.\arabic{equation}}}
\numberlikeadb
\newtheorem{theorem}{Theorem}[section]
\newtheorem{lemma}[theorem]{Lemma}
\newtheorem{corollary}[theorem]{Corollary}
\newtheorem{open}[theorem]{Open Problem}
\newtheorem{conj}[theorem]{Conjecture}

\newtheorem{remark}[theorem]{Remark}

\usepackage{lscape}
\usepackage{caption}
\usepackage{multirow}
\begin{document}

\title{Inequalities for some integrals involving modified Bessel functions}
\author{Robert E. Gaunt\footnote{School of Mathematics, The University of Manchester, Manchester M13 9PL, UK}}

\date{\today} 
\maketitle

\vspace{-5mm}

\begin{abstract}Simple inequalities are established for some integrals involving the modified Bessel functions of the first and second kind.  In most cases these inequalities are tight in certain limits.  As a consequence, we deduce a tight double inequality, involving the modified Bessel function of the first kind, for a generalized hypergeometric function.  We also present some open problems that arise from this research.
\end{abstract}


\noindent{{\bf{Keywords:}}} Modified Bessel function; inequality; integral

\noindent{{{\bf{AMS 2010 Subject Classification:}}} Primary 33C10; 26D15

\section{Introduction}\label{intro}

In the recent papers \cite{gaunt ineq1} and \cite{gaunt ineq3}, simple inequalities, involving modified Bessel functions of the first kind $I_\nu(x)$ and second kind $K_\nu(x)$, were obtained for the integrals
\begin{equation}\label{intbes} \int_0^x \mathrm{e}^{-\gamma t} t^\nu I_\nu(t)\,\mathrm{d}t, \qquad \int_x^\infty \mathrm{e}^{\gamma t}t^\nu K_\nu(t)\,\mathrm{d}t,
\end{equation}
where $x>0$, $0\leq\gamma<1$ and $\nu>-\frac{1}{2}$.  For $\gamma\not=0$ there does not exist simple closed form expressions for these integrals.  Similar inequalities for integrals involving the modified Struve function of the first kind $\mathbf{L}_\nu(x)$ have also been established by \cite{gaunt ineq4}.  The bounds of \cite{gaunt ineq1,gaunt ineq3} were required in the development of Stein's method \cite{stein,chen,np12} for variance-gamma approximation \cite{eichelsbacher, gaunt vg, gaunt vg2}. Although, the combination of their simple form and accuracy mean that the inequalities may also prove useful in other problems involving modified Bessel functions; see, for example, \cite{bs09,baricz3} in which inequalities for the modified Bessel function of the first kind are used to obtain tight bounds for the generalized Marcum Q-function, which frequently arises in radar signal processing.

In this paper, we consider the problem of obtaining inequalities, involving modified Bessel functions, for the integrals
\begin{equation}\label{intstruve} \int_0^x \mathrm{e}^{-\gamma t} \frac{I_\nu(t)}{t^\nu}\,\mathrm{d}t, \qquad \int_x^\infty \mathrm{e}^{\gamma t}\frac{ K_\nu(t)}{t^\nu}\,\mathrm{d}t,
\end{equation}
where $x>0$ and $0\leq\gamma<1$. The conditions imposed on $\nu$ will be different for several of the inequalities.  (In fact, most of our inequalities are for integrals of a more general form.)  This is a related problem to the one considered by \cite{gaunt ineq1,gaunt ineq3}, and, as is the case for the integrals of (\ref{intbes}), it is natural to consider such integrals. Indeed, several related integrals are tabulated in standard references.  For example, formulas 10.43.8 and 10.43.10 of \cite{olver} are
\begin{align*}\int_0^x \mathrm{e}^{-t}\frac{I_\nu(t)}{t^\nu}\,\mathrm{d}t&=\frac{2^{-\nu+1}}{(2\nu-1)\Gamma(\nu)}-\frac{\mathrm{e}^{-x}x^{-\nu+1}}{2\nu-1}\big(I_\nu(x)+I_{\nu-1}(x)\big), \quad \nu\notin\{\tfrac{1}{2}\}\cup-\mathbb{N}_0, \\
\int_x^\infty \mathrm{e}^t \frac{K_{\nu}(t)}{t^\nu}\,\mathrm{d}t&=\frac{\mathrm{e}^x x^{-\nu+1}}{2\nu-1}\big(K_{\nu}(x)+K_{\nu-1}(x)\big),\quad \nu>\tfrac{1}{2}.
\end{align*} 
When $\gamma=0$ the integrals in (\ref{intstruve}) can also be evaluated because the modified Bessel functions can be represented through the generalized hypergeometric function.  To see this, recall that the generalized hypergeometric function (see \cite{olver}) is defined by
\begin{equation*}{}_pF_q\big(a_1,\ldots,a_p;b_1,\ldots,b_q;x\big)=\sum_{k=0}^\infty \frac{(a_1)_k\cdots(a_p)_k}{(b_1)_k\cdots(b_q)_k}\frac{x^k}{k!},
\end{equation*}
and the Pochhammer symbol is given by $(a)_0=1$ and $(a)_k=a(a+1)(a+2)\cdots(a+k-1)$, $k\geq1$.  The modified Bessel functions can be defined for $x>0$ and $\nu\in\mathbb{R}$ by
\begin{equation}\label{defI}I_\nu(x)=\sum_{k=0}^\infty\frac{(\frac{1}{2}x)^{\nu+2k}}{k!\Gamma(\nu+k+1)}, \qquad K_\nu(x)=\frac{\pi}{2\sin(\nu\pi)}\big(I_{-\nu}(x)-I_\nu(x)\big),
\end{equation}
where for $K_\nu(x)$ a limit is taken for integer $\nu$.  We then have the representation (see formula 10.39.9 of \cite{olver})
\begin{equation*}I_\nu(x)=\frac{(\frac{1}{2}x)^\nu}{\Gamma(\nu+1)} {}_0F_1\bigg(-;\nu+1;\frac{x^2}{4}\bigg),
\end{equation*}
and straightforward calculations then yield
\begin{align}\label{besint6}\int \frac{I_\nu(t)}{t^\nu}\,\mathrm{d}t&=\frac{x}{2^\nu\Gamma(\nu+1)}{}_1F_2\bigg(\frac{1}{2};\frac{3}{2},\nu+1;\frac{x^2}{4}\bigg), \\
\int \frac{K_\nu(t)}{t^\nu}\,\mathrm{d}t&=-\frac{\pi x}{2^{\nu+1}\sin(\pi\nu)}\bigg\{\frac{4^\nu x^{-2\nu}}{(2\nu-1)\Gamma(1-\nu)}{}_1F_2\bigg(\frac{1}{2}-\nu;\frac{3}{2}-\nu,1-\nu;\frac{x^2}{4}\bigg)\nonumber\\
\label{besint89}&\quad+\frac{1}{\Gamma(\nu+1)}{}_1F_2\bigg(\frac{1}{2};\frac{3}{2},\nu+1;\frac{x^2}{4}\bigg)\bigg\}.
\end{align}
However, when $\gamma\not=0$ there does not exist closed form formulas for the integrals in (\ref{intstruve}).  Moreover, even when $\gamma=0$ the integrals are given in terms of the generalized hypergeometric function.  This provides the motivation for establishing simple bounds, involving modified Bessel functions, for these integrals.

The approach used in this paper to bound the integrals in (\ref{intstruve}), which involves exploiting basic identities and monotonicity properties of modified Bessel functions, is similar to the one used in \cite{gaunt ineq1,gaunt ineq3}.  These elementary properties of modified Bessel functions are collected in Appendix \ref{appb}.  In Appendix \ref{appa}, we prove an inequality involving the modified Bessel function $K_\nu(x)$ that is needed to prove one of our integral inequalities.  This result may be of independent interest.   The inequalities obtained in this paper are simple, but, in most cases, are seen to be tight in certain limits.


\section{Inequalities for integrals of modified Bessel functions}\label{sec2}

The following theorem complements the inequalities of Theorem 2.5 of \cite{gaunt ineq1} and Theorem 2.2 of \cite{gaunt ineq3} for the integral $\int_x^\infty \mathrm{e}^{\gamma t}t^\nu K_\nu(t)\,\mathrm{d}t$. 

\begin{theorem}\label{tiger}Let $0<\gamma<1$ and $n<1$.    Then, for all $x>0$,
\begin{align}\label{bk1}\int_x^\infty\frac{K_{\nu+n}(t)}{t^\nu}\,\mathrm{d}t&>\frac{K_{\nu+n-2}(x)}{x^\nu}, \quad \nu>\tfrac{5}{2}-2n, \\
\label{bk3}\int_x^\infty \frac{K_{\nu+n}(t)}{t^\nu}\,\mathrm{d}t&<\frac{K_{\nu+n-1}(x)}{x^\nu}, \quad \nu\in\mathbb{R}, \\
\label{bk2}\int_x^\infty \frac{K_{\nu+n}(t)}{t^\nu}\,\mathrm{d}t&<\frac{2(\nu+n-1)}{2\nu+n-1}\frac{K_{\nu+n-1}(x)}{x^\nu}-\frac{n-1}{2\nu+n-1}\frac{K_{\nu+n-3}(x)}{x^\nu}, \\ 
&\quad\quad\quad\quad\quad\quad\quad\quad\quad\quad\quad\quad\quad\quad\quad\quad\quad\quad\quad \text{for $\nu>\tfrac{1}{2}(1-n)$.} \nonumber \\
\label{bk4}\int_x^\infty \mathrm{e}^{\gamma t}\frac{K_{\nu+n}(t)}{t^\nu}\,\mathrm{d}t&<\frac{\mathrm{e}^{\gamma x}}{1-\gamma}\int_x^\infty\frac{K_{\nu+n}(t)}{t^\nu}\,\mathrm{d}t, \quad \nu\geq\tfrac{1}{2}-n, \\
\label{bk6}\int_x^\infty \mathrm{e}^{\gamma t}\frac{K_{\nu+n}(t)}{t^\nu}\,\mathrm{d}t&<\frac{\mathrm{e}^{\gamma x}}{1-\gamma}\frac{K_{\nu+n-1}(x)}{x^\nu}, \quad \nu\geq\tfrac{1}{2}-n,  \\
\label{bk5}\int_x^\infty \mathrm{e}^{\gamma t}\frac{K_{\nu+n}(t)}{t^\nu}\,\mathrm{d}t&<\frac{\mathrm{e}^{\gamma x}}{1-\gamma}\bigg(\frac{2(\nu+n-1)}{2\nu+n-1}\frac{K_{\nu+n-1}(x)}{x^\nu}-\frac{n-1}{2\nu+n-1}\frac{K_{\nu+n-3}(x)}{x^\nu}\bigg), \\
&\quad\quad\quad\quad\quad\quad\quad\quad\quad\quad\quad\quad\quad\quad\quad\quad\quad\quad\quad \text{for $\nu>\tfrac{1}{2}(1-n)$.} \nonumber
\end{align}
Inequality (\ref{bk3}) is reversed when $n>1$ and equality is attained when $n=1$.  Equality is attained in (\ref{bk4}) and (\ref{bk6}) when $n=1$ and $\nu=-\tfrac{1}{2}$, and the inequalities are reversed if $n>1$ and $\nu\leq\frac{1}{2}-n$.  Inequalities (\ref{bk1})--(\ref{bk5}) are tight as $x\rightarrow\infty$.  Inequality (\ref{bk2}) is also tight as $x\downarrow0$, provided $\nu>1-n$.
\end{theorem}


\begin{proof}We first establish inequalities (\ref{bk1})--(\ref{bk5}) and then prove that the inequalities are tight in certain limits.

(i)  Let us first compare the derivatives of $\int_x^\infty\frac{K_{\nu+n}(t)}{t^\nu}\,\mathrm{d}t$ and $\frac{K_{\nu+n-2}(x)}{x^\nu}$.  From the differentiation formula (\ref{diffKi}) and identity (\ref{Kidentity}) we have
\begin{align}\frac{\mathrm{d}}{\mathrm{d}x}\bigg(\frac{K_{\nu+n-2}(x)}{x^\nu}\bigg)&=\frac{\mathrm{d}}{\mathrm{d}x}\bigg(x^{n-2}\cdot\frac{K_{\nu+n-2}(x)}{x^{\nu+n-2}}\bigg)=(n-2)\frac{K_{\nu+n-2}(x)}{x^{\nu+1}}-\frac{K_{\nu+n-1}(x)}{x^\nu} \nonumber\\
&=\frac{n-2}{2(\nu+n-2)}\bigg(\frac{K_{\nu+n-1}(x)}{x^\nu}-\frac{K_{\nu+n-3}(x)}{x^\nu}\bigg)-\frac{K_{\nu+n-1}(x)}{x^\nu}\nonumber \\
\label{jpjp}&=-\frac{2\nu+n-2}{2(\nu+n-2)}\frac{K_{\nu+n-1}(x)}{x^\nu}+\frac{2-n}{2(\nu+n-2)}\frac{K_{\nu+n-3}(x)}{x^\nu} \\
\label{lmbv}&>-\frac{K_{\nu+n}(x)}{x^\nu}=\frac{\mathrm{d}}{\mathrm{d}x}\bigg(\int_x^\infty\frac{K_{\nu+n}(t)}{t^\nu}\,\mathrm{d}t\bigg),
\end{align} 
where the inequality follows from Lemma \ref{lemap}. Thus,  $\int_x^\infty\frac{K_{\nu+n}(t)}{t^\nu}\,\mathrm{d}t$ decays at a faster rate than $\frac{K_{\nu+n-2}(x)}{x^\nu}$ for all $x>0$.  

We now consider the asymptotic behaviour of $\int_x^\infty\frac{K_{\nu+n}(t)}{t^\nu}\,\mathrm{d}t$ and $\frac{K_{\nu+n-2}(x)}{x^\nu}$ as $x\rightarrow\infty$.  To this end, we note that a routine asymptotic analysis using integration by parts gives that, as $x\rightarrow\infty$,
\begin{align*}\int_x^\infty t^{-\nu-1/2}\mathrm{e}^{-t}\,\mathrm{d}t&\sim x^{-\nu-1/2}\mathrm{e}^{-x}\bigg(1-\frac{\nu+\tfrac{1}{2}}{x}\bigg), \\
\int_x^\infty t^{-\nu-3/2}\mathrm{e}^{-t}\,\mathrm{d}t&\sim x^{-\nu-3/2}\mathrm{e}^{-x}.
\end{align*}
Using  (\ref{Ktendinfinity}) and these limiting forms gives that, as $x\rightarrow\infty$,
\begin{align}\int_x^\infty\frac{K_{\nu+n}(t)}{t^\nu}\,\mathrm{d}t&\sim\sqrt{\frac{\pi}{2}}\int_x^\infty t^{-\nu-1/2}\mathrm{e}^{-t}\bigg(1+\frac{4(\nu+n)^2-1}{8t}\bigg)\,\mathrm{d}t \nonumber\\
\label{cheb}&\sim\sqrt{\frac{\pi}{2}}x^{-\nu-1/2}\mathrm{e}^{-x}\bigg(1+\frac{4(\nu+n)^2-8\nu-5}{8x}\bigg).
\end{align}
We also have, as $x\rightarrow\infty$,
\begin{align}\label{chec}\frac{K_{\nu+n-2} (x)}{x^\nu} \sim \sqrt{\frac{\pi}{2}}x^{-\nu-1/2} \mathrm{e}^{-x}\bigg(1+\frac{4(\nu+n-2)^2-1}{8x}\bigg).
\end{align}
One can readily check that the second term in the expansion (\ref{cheb}) is greater than the second term in the expansion (\ref{chec}) if $\nu>\frac{5}{2}-2n$.  Combining this with (\ref{lmbv}) then proves inequality (\ref{bk1}).

(ii) Suppose $n<1$.  Then, for $x>0$,
\begin{equation*}\int_x^\infty \frac{K_{\nu+n}(t)}{t^\nu}\,\mathrm{d}t=\int_x^\infty t^{n-1}\cdot \frac{K_{\nu+n}(t)}{t^{\nu+n-1}}\,\mathrm{d}t<x^{n-1}\int_x^\infty\frac{K_{\nu+n}(t)}{t^{\nu+n-1}}\,\mathrm{d}t=\frac{K_{\nu+n-1}(x)}{x^\nu},
\end{equation*} 
where we used (\ref{diffKi}) to evaluate the integral. We can see that we have equality if $n=1$ and the inequality is reversed if $n>1$.

(iii) From (\ref{jpjp}), we have
\begin{equation*}\frac{\mathrm{d}}{\mathrm{d}t}\bigg(\frac{K_{\nu+n-1}(t)}{t^\nu}\bigg)=-\frac{2\nu+n-1}{2(\nu+n-1)}\frac{K_{\nu+n}(t)}{t^\nu}-\frac{n-1}{2(\nu+n-1)}\frac{K_{\nu+n-2}(t)}{t^\nu}.
\end{equation*}
Integrating both sides of this expression over $(x,\infty)$, applying the fundamental theorem of calculus and rearranging gives
\begin{equation*}\int_x^\infty\frac{K_{\nu+n}(t)}{t^\nu}\,\mathrm{d}t=\frac{2(\nu+n-1)}{2\nu+n-1}\frac{K_{\nu+n-1}(x)}{x^\nu}-\frac{n-1}{2\nu+n-1}\int_x^\infty\frac{K_{\nu+n-2}(t)}{t^\nu}\,\mathrm{d}t.
\end{equation*}
On applying inequality (\ref{bk3}) to the integral on the right hand-side of the above expression we obtain (\ref{bk2}).

(iv) Suppose $n<1$ and $\nu\geq\tfrac{1}{2}-n$.  Using integration by parts gives
\begin{align*}\int_x^\infty \mathrm{e}^{\gamma t}\frac{K_{\nu+n}(t)}{t^\nu}\,\mathrm{d}t&=\mathrm{e}^{\gamma x}\int_x^\infty \frac{K_{\nu+n}(t)}{t^\nu}\,\mathrm{d}t +\gamma\int_x^\infty \mathrm{e}^{\gamma t}\bigg(\int_t^\infty\frac{K_{\nu+n}(u)}{u^\nu}\,\mathrm{d}u\bigg)\,\mathrm{d}t \\
&< \mathrm{e}^{\gamma x}\int_x^\infty \frac{K_{\nu+n}(t)}{t^\nu}\,\mathrm{d}t +\gamma\int_x^\infty \mathrm{e}^{\gamma t}\frac{K_{\nu+n-1}(t)}{t^\nu}\,\mathrm{d}t \\
&\leq \mathrm{e}^{\gamma x}\int_x^\infty \frac{K_{\nu+n}(t)}{t^\nu}\,\mathrm{d}t +\gamma\int_x^\infty \mathrm{e}^{\gamma t}\frac{K_{\nu+n}(t)}{t^\nu}\,\mathrm{d}t,
\end{align*}
where we used (\ref{bk3}) and (\ref{cake}) (as $\nu\geq\frac{1}{2}-n$) to obtain the first and second inequalities respectively.  Rearranging yields inequality (\ref{bk4}).  If $n=1$, then the strict inequality in the above expression because an equality.  If also $\nu=-\frac{1}{2}$, then the final inequality becomes an equality.  Therefore if $n=1$ and $\nu=-\frac{1}{2}$ then (\ref{bk4}) is an equality.  Similar considerations, using (\ref{cakkk}), show that (\ref{bk4}) is reversed if $n>1$ and $\nu\leq\frac{1}{2}-n$.

(v) Combine inequalities (\ref{bk4}) and (\ref{bk3}).

(vi) Combine inequalities (\ref{bk4}) and (\ref{bk2}).

(vii) To see that inequalities (\ref{bk1})--(\ref{bk5}) are tight as $x\rightarrow\infty$, we first note that using the limiting form (\ref{Ktendinfinity}) followed by a straightforward asymptotic analysis gives that, for $0\leq\gamma<1$, $n\in\mathbb{R}$ and $\nu\in\mathbb{R}$,
\begin{equation*}\int_x^\infty \mathrm{e}^{\gamma t}\frac{K_{\nu+n}(t)}{t^\nu}\,\mathrm{d}t \sim\sqrt{\frac{\pi}{2}}\int_x^\infty t^{-\nu-1/2}\mathrm{e}^{-(1-\gamma)t}\,\mathrm{d}t\sim\sqrt{\frac{\pi}{2}}\frac{1}{1-\gamma}x^{-\nu-1/2}\mathrm{e}^{-(1-\gamma)x},  
\end{equation*}
as $x\rightarrow\infty$.  As an example, inequality (\ref{bk6}) is tight as $x\rightarrow\infty$ because in this limit we also have
\begin{equation*}\frac{\mathrm{e}^{\gamma x}}{1-\gamma}\frac{K_{\nu+n-1}(x)}{x^\nu}\sim\sqrt{\frac{\pi}{2}}\frac{1}{1-\gamma}x^{-\nu-1/2}\mathrm{e}^{-(1-\gamma)x},
\end{equation*}
and the tightness of the other inequalities is established similarly.

Establishing that inequality (\ref{bk2}) is tight as $x\downarrow0$, provided $\nu>1-n$ (which we now assume from here on), is a little more involved.  We first find a limiting form for the integral in the limit $x\downarrow0$.  Suppose that $0<x\ll y\ll 1$.  We may write
\begin{equation*}\int_x^\infty \frac{K_{\nu+n}(t)}{t^\nu}\,\mathrm{d}t=\int_x^y \frac{K_{\nu+n}(t)}{t^\nu}\,\mathrm{d}t+\int_y^\infty \frac{K_{\nu+n}(t)}{t^\nu}\,\mathrm{d}t =: J_1+J_2.
\end{equation*}
Then, from the limiting form (\ref{Ktend0}), we have
\begin{align*}J_1\sim \int_x^y 2^{\nu+n-1}\Gamma(\nu+n)t^{-2\nu-n}\,\mathrm{d}t &=\frac{2^{\nu+n-1}\Gamma(\nu+n)}{2\nu+n-1}\big(x^{-2\nu-n+1}+y^{-2\nu-n+1}\big)\\
&\sim \frac{2^{\nu+n-1}\Gamma(\nu+n)}{2\nu+n-1}x^{-2\nu-n+1},
\end{align*} 
and, on using inequality (\ref{bk3}) to bound the integral,
\begin{align*}J_2=\int_y^\infty\frac{K_{\nu+n}(t)}{t^\nu}\,\mathrm{d}t<\frac{K_{\nu+n-1}(y)}{y^\nu} \ll\frac{K_{\nu+n-1}(x)}{x^\nu}\sim \frac{2^{\nu+n-2}\Gamma(\nu+n-1)}{x^{2\nu+n-1}}.
\end{align*}
Therefore $J_2\ll J_1$, and we have
\begin{equation*}\int_x^\infty \frac{K_{\nu+n}(t)}{t^\nu}\,\mathrm{d}t\sim\frac{2^{\nu+n-1}\Gamma(\nu+n)}{2\nu+n-1}x^{-2\nu-n+1}.
\end{equation*}
But, from (\ref{Ktend0}) and the standard formula $u\Gamma(u)=\Gamma(u+1)$, we also have
\begin{align*}\frac{2(\nu+n-1)}{2\nu+n-1}\frac{K_{\nu+n-1}(x)}{x^\nu}-\frac{n-1}{2\nu+n-1}\frac{K_{\nu+n-3}(x)}{x^\nu}\sim\frac{2^{\nu+n-1}\Gamma(\nu+n)}{2\nu+n-1}x^{-2\nu-n+1},
\end{align*}
from which it follows that inequality (\ref{bk2}) is tight as $x\downarrow0$, provided $\nu>1-n$.
\end{proof}


\begin{remark}Inequality (\ref{bk4}) bounds the integral $\int_x^\infty \mathrm{e}^{\gamma t}\frac{K_{\nu+n}(t)}{t^\nu}\,\mathrm{d}t$ in terms of the easier to bound integral $\int_x^\infty \frac{K_{\nu+n}(t)}{t^\nu}\,\mathrm{d}t$.  In this way inequality (\ref{bk4}) is used to deduce inequalities (\ref{bk6}) and (\ref{bk5}) from inequalities (\ref{bk3}) and (\ref{bk2}).  Thus, inequality (\ref{bk4}) is more accurate than inequalities (\ref{bk6}) and (\ref{bk5}), but suffers from taking a more complicated form.  A similar comment applies to inequalities (\ref{bi4}) and (\ref{bi5}) of Theorem \ref{tiger1} below.

Also, it is worth noting that, due to the positivity of $K_\nu(x)$ for $x>0$ and $\nu\in\mathbb{R}$, inequality (\ref{bk2}) outperforms inequality (\ref{bk3}) if $\frac{n-1}{2\nu+n-1}>0$, with the reverse being true if $\frac{n-1}{2\nu+n-1}<0$.  A similar comparison can be made between inequalities (\ref{bk6}) and (\ref{bk5}).
\end{remark}

\begin{remark}Arguing similarly to we did in part (vii) of the proof of Theorem \ref{tiger}, we have, as $x\downarrow0$
\begin{equation}\label{bfgh}\int_x^\infty \mathrm{e}^{\gamma t}\frac{K_{\nu+n}(t)}{t^\nu}\,\mathrm{d}t\sim\frac{2^{\nu+n-1}\Gamma(\nu+n)}{2\nu+n-1}x^{-2\nu-n+1}.
\end{equation}
Here the limiting form does not involve $\gamma$, and inequality (\ref{bk5}) is thus not tight as $x\downarrow0$.  This is in contrast to inequality (\ref{bk2}), which can be obtained by setting $\gamma=0$ in (\ref{bk5}).  

Now let $\gamma>0$ (note that here we are not imposing that $\gamma<1$).  Then we can use the fact that $\mathrm{e}^{\gamma t}$ is an increasing function of $t$ to obtain from inequality (\ref{bk1}) that, for $x>0$,
\begin{equation*}\int_x^\infty \mathrm{e}^{\gamma t}\frac{K_{\nu+n}(t)}{t^\nu}\,\mathrm{d}t>\mathrm{e}^{\gamma x}\frac{K_{\nu+n-2}(x)}{x^\nu}, \quad \nu>\tfrac{5}{2}-2n, \:\: n<1.
\end{equation*}
Unlike inequalities (\ref{bk4})--(\ref{bk5}), this bound does not contain a $\frac{1}{1-\gamma}$ factor and is therefore not tight as $x\rightarrow\infty$.   
\end{remark}

This author believes that inequality(\ref{bk1}) can be improved to $\int_x^\infty\frac{K_{\nu+n}(t)}{t^\nu}\,\mathrm{d}t>\frac{K_{\nu+n-\beta_{\nu,n}}(x)}{x^\nu}$, where $\beta_{\nu,n}<2$ for $\nu>\frac{5}{2}-2n$, $n<1$.  That this is an improvement can be seen from inequality (\ref{cakkk}).   

\begin{conj}\label{connn}Let $n<1$.  For $\nu\geq 1-n+\sqrt{2(1-n)}$, define $\alpha_{\nu,n}:=\nu+n-\sqrt{(\nu+n)^2-2\nu-1}$.  Then, for all $x>0$,
\begin{equation}\label{bk100}\int_x^\infty\frac{K_{\nu+n}(t)}{t^\nu}\,\mathrm{d}t>\frac{K_{\nu+n-\alpha_{\nu,n}}(x)}{x^\nu}, \quad \nu\geq1-n+\sqrt{2(1-n)}.
\end{equation}
We have equality in (\ref{bk100}) when $n=1$, and this equality is valid for all $\nu\in\mathbb{R}$.  

Let $\alpha_{\nu,n}':=\nu+n+\sqrt{(\nu+n)^2-2\nu-1}$, so that $\nu+n-\alpha_{\nu,n}'=-(\nu+n-\alpha_{\nu,n})$ and thus $K_{\nu+n-\alpha_{\nu,n}'}(x)=K_{\nu+n-\alpha_{\nu,n}}(x)$.  Then, the index $\nu+n-\alpha_{\nu,n}$ (and equivalently $\nu+n-\alpha_{\nu,n}'$) in inequality (\ref{bk100}) is best possible, in the sense that, for any $\beta\notin\{\alpha_{\nu,n},\alpha_{\nu,n}'\}$, either $\frac{K_{\nu+n-\beta}(x)}{x^\nu}<\frac{K_{\nu+n-\alpha_{\nu,n}}(x)}{x^\nu}$ for all $x>0$, or there exists a $y>0$ such that $\int_y^\infty\frac{K_{\nu+n}(t)}{t^\nu}\,\mathrm{d}t<\frac{K_{\nu+n-\beta}(y)}{y^\nu}$.
\end{conj}

\begin{remark}In interpreting Conjecture \ref{connn}, it useful to note the following.  We have
\[\frac{\partial \alpha_{\nu,n}}{\partial \nu}=1-\frac{\nu+n-1}{\sqrt{(\nu+n)^2-2\nu-1}},\]
and one can readily check that, for $\nu$ and $n$ defined as in Conjecture \ref{connn}, this derivative is negative.  Therefore, for fixed $n$, $\alpha_{\nu,n}$ is a decreasing function of $\nu$, for $\nu\geq1-n+\sqrt{2(1-n)}$.  This leads to the bound $1\leq\alpha_{\nu,n}\leq 1+\sqrt{2(1-n)}$.  Moreover,  $\alpha_{\nu,n}<2$ for $\nu>\frac{5}{2}-2n$, $n<1$.

Let us now sketch why it reasonable to believe that Conjecture \ref{connn} may hold.  Comparing the $x\rightarrow\infty$ asymptotic expansions of $\int_x^\infty\frac{K_{\nu+n}(t)}{t^\nu}\,\mathrm{d}t$ and $\frac{K_{\nu+n-\beta}(x)}{x^\nu}$, one can see that the first terms agree for both functions, the second terms agree if and only if $\beta\in\{\alpha_{\nu,n},\alpha_{\nu,n}'\}$, and for such $\beta$ the third term in the expansion of $\int_x^\infty\frac{K_{\nu+n}(t)}{t^\nu}\,\mathrm{d}t$ is greater than that of $\frac{K_{\nu+n-\beta}(x)}{x^\nu}$.  (This analysis also leads to the assertion of optimally of $\alpha_{\nu,n}$ and $\alpha_{\nu,n}'$, as stated in the conjecture.)  Also, numerical experiments carried out using \emph{Mathematica} suggest that the inequality below is valid for $\beta=\alpha_{\nu,n}$, provided $\nu$ and $n$ are as defined in Conjecture \ref{connn}:
\begin{align}\frac{\mathrm{d}}{\mathrm{d}x}\bigg(\frac{K_{\nu+n-\beta}(x)}{x^\nu}\bigg)&=-\frac{2\nu+n-\beta}{2(\nu+n-\beta)}\frac{K_{\nu+n-\beta+1}(x)}{x^\nu}-\frac{n-\beta}{2(\nu+n-\beta)}\frac{K_{\nu+n-\beta-1}(x)}{x^\nu}\nonumber \\
\label{bk99}&\stackrel{(?)}{>}-\frac{K_{\nu+n}(x)}{x^\nu}=\frac{\mathrm{d}}{\mathrm{d}x}\bigg(\int_x^\infty\frac{K_{\nu+n}(t)}{t^\nu}\,\mathrm{d}t\bigg).
\end{align}
(The case $\beta=2$ was considered in the proof of inequality (\ref{bk1}).)  If inequality (\ref{bk99}) was proved rigorously, for $\beta=\alpha_{\nu,n}$, then arguing similarly to we did in the proof of inequality (\ref{bk1}) would prove inequality (\ref{bk100}).  Our proof of inequality (\ref{bk99}) for the case $\beta=2$ (see Appendix \ref{appa}) relied heavily on the fact that this value of $\beta$ is an integer, meaning that it cannot be easily adapted to $\beta=\alpha_{\nu,n}$.
\end{remark}

The inequalities in the following theorem complement the inequalities for the integral $\int_0^x \mathrm{e}^{-\gamma t}t^\nu I_\nu(t)\,\mathrm{d}t$ that are given in Theorem 2.1 of \cite{gaunt ineq1} and Theorem 2.3 of \cite{gaunt ineq3}.

\begin{theorem}\label{tiger1}Let $0<\gamma<1$ and $n>-1$. Then, for all $x>0$,
\begin{align}\label{bi1}\int_0^x \frac{I_\nu(t)}{t^\nu}\,\mathrm{d}t&>\frac{I_\nu(x)}{x^\nu}-\frac{1}{2^\nu\Gamma(\nu+1)}, \quad \nu>-1, \\
\label{bi2}\int_0^x \frac{I_{\nu+n}(t)}{t^\nu}\,\mathrm{d}t&>\frac{I_{\nu+n+1}(x)}{x^\nu}, \quad \nu>-\tfrac{1}{2}(n+1), \\
\label{bi3}\int_0^x \frac{I_{\nu+n}(t)}{t^\nu}\,\mathrm{d}t&<\frac{2(\nu+n+1)}{n+1}\frac{I_{\nu+n+1}(x)}{x^\nu}-\frac{2\nu+n+1}{n+1}\frac{I_{\nu+n+3}(x)}{x^\nu}, \: \nu>-\tfrac{1}{2}(n+1), \\
\label{bi4}\int_0^x \mathrm{e}^{-\gamma t}\frac{I_\nu(t)}{t^\nu}\,\mathrm{d}t&>\frac{1}{1-\gamma}\bigg(\mathrm{e}^{-\gamma x}\int_0^x\frac{I_\nu(t)}{t^\nu}\,\mathrm{d}t-\frac{1}{2^\nu\Gamma(\nu+1)}(1-\mathrm{e}^{-\gamma x})\bigg), \quad \nu>-1, \\
\label{bi5}\int_0^x \mathrm{e}^{-\gamma t}\frac{I_\nu(t)}{t^\nu}\,\mathrm{d}t&>\frac{1}{1-\gamma}\bigg(\mathrm{e}^{-\gamma x}\frac{I_{\nu}(x)}{x^\nu}-\frac{1}{2^\nu\Gamma(\nu+1)}\bigg), \quad \nu>-1.
\end{align}
We have equality in (\ref{bi2}) and (\ref{bi3}) if $\nu=-\frac{1}{2}(n+1)$.  Inequalities (\ref{bi1})--(\ref{bi5}) are tight as $x\rightarrow\infty$ and inequality (\ref{bi3}) is also tight as $x\downarrow0$.   

Now suppose that $\nu>-\frac{1}{2}(n+1)$, and let 
\begin{equation*}C_{\nu,n}:=\sup_{x>0}\frac{x^\nu}{I_{\nu+n}(x)}\int_0^x \frac{I_{\nu+n}(t)}{t^\nu}\,\mathrm{d}t.
\end{equation*}
The existence of $C_{\nu,n}$ is guaranteed by inequalities (\ref{bi3}) and (\ref{Imon}), and we have $C_{\nu,n}<2(\nu+n+1)$.  Suppose also that $0<\gamma<\frac{1}{C_{\nu,n}}$.  Then, for all $x>0$,
\begin{align}\label{bi7}\int_0^x \mathrm{e}^{-\gamma t}\frac{I_{\nu+n}(t)}{t^\nu}\,\mathrm{d}t&<\frac{\mathrm{e}^{-\gamma x}}{1-C_{\nu,n}\gamma}\int_0^x\frac{I_{\nu+n}(t)}{t^\nu}\,\mathrm{d}t, \\
\label{bi8}\int_0^x \mathrm{e}^{-\gamma t}\frac{I_{\nu+n}(t)}{t^\nu}\,\mathrm{d}t&< \frac{\mathrm{e}^{-\gamma x}}{1-C_{\nu,n}\gamma}\bigg(\frac{2(\nu+n+1)}{n+1}\frac{I_{\nu+n+1}(x)}{x^\nu}-\frac{2\nu+n+1}{n+1}\frac{I_{\nu+n+3}(x)}{x^\nu} \bigg).
\end{align}
\end{theorem}

\begin{proof}We proceed as in the proof of Theorem \ref{tiger} by first establishing inequalities (\ref{bi1})--(\ref{bi8}) and then proving that the inequalities are tight in certain limits.

(i) From inequality (\ref{Imon}) and the differentiation formula (\ref{diffone}) we obtain
\begin{align*}\int_0^x \frac{I_\nu(t)}{t^\nu}\,\mathrm{d}t>\int_0^x \frac{I_{\nu+1}(t)}{t^\nu}\,\mathrm{d}t=\frac{I_\nu(x)}{x^\nu}-\lim_{x\downarrow0}\frac{I_\nu(x)}{x^\nu} =\frac{I_\nu(x)}{x^\nu}-\frac{1}{2^\nu\Gamma(\nu+1)},
\end{align*}
where we used the limiting form (\ref{Itend0}) in the final step.

(ii) The assertion that (\ref{bi2}) is an equality if $\nu=-\frac{1}{2}(n+1)$ can be readily checked using (\ref{diffone}) and (\ref{Itend0}).  The same applies to (\ref{bi3}).   We now suppose that $\nu>-\frac{1}{2}(n+1)$.  Consider the function
\begin{equation*}u(x)=\int_0^x \frac{I_{\nu +n}(t)}{t^\nu}\,\mathrm{d}t-\frac{I_{\nu+n+1}(x)}{x^\nu}.
\end{equation*}
We argue that $u(x)>0$ for all $x>0$, which will prove the result.  We first note that from the differentiation formula (\ref{diffone}) followed by identity (\ref{Iidentity}) we have that
\begin{align}\frac{\mathrm{d}}{\mathrm{d}x}\bigg(\frac{I_{\nu+n+1}(x)}{x^\nu}\bigg)&=\frac{\mathrm{d}}{\mathrm{d}x}\bigg(x^{n+1}\cdot\frac{I_{\nu+n+1}(x)}{x^{\nu+n+1}}\bigg)=(n+1)\frac{I_{\nu+n+1}(x)}{x^{\nu+1}}+\frac{I_{\nu+n+2}(x)}{x^{\nu}}\nonumber\\
&=\frac{n+1}{2(\nu+n+1)}\bigg(\frac{I_{\nu+n}(x)}{x^\nu}-\frac{I_{\nu+n+2}(x)}{x^\nu}\bigg)+\frac{I_{\nu+n+2}(x)}{x^\nu}\nonumber \\
\label{num3}&=\frac{n+1}{2(\nu+n+1)}\frac{I_{\nu+n}(x)}{x^\nu}+\frac{2\nu+n+1}{2(\nu+n+1)}\frac{I_{\nu+n+2}(x)}{x^\nu}.
\end{align}
Therefore
\begin{align*}u'(x)=\frac{2\nu+n+1}{2(\nu+n+1)}\bigg(\frac{I_{\nu+n}(x)}{x^\nu}-\frac{I_{\nu+n+2}(x)}{x^\nu}\bigg) >0,
\end{align*}
where we used (\ref{Imon}) to obtain the inequality.  Also, from (\ref{Itend0}), as $x\downarrow0$,
\begin{align*}u(x)&\sim \int_0^x \frac{t^{n}}{2^{\nu+n}\Gamma(\nu+n+1)}\,\mathrm{d}t-\frac{x^{n+1}}{2^{\nu+n+1}\Gamma(\nu+n+2)}\\
&=\frac{x^{n+1}}{2^{\nu+n}(n+1)\Gamma(n+\nu+1)} -\frac{x^{n+1}}{2^{\nu+n+1}\Gamma(\nu+n+2)}\\ 
&=\frac{x^{n+1}}{2^{\nu+n}\Gamma(\nu+n+1)}\bigg(\frac{1}{n+1}-\frac{1}{2(\nu+n+1)}\bigg) >0,
\end{align*}
where the inequality holds because $\nu>-\frac{1}{2}(n+1)$.  Thus, we conclude that $u(x)>0$ for all $x>0$, as required.

(iii) Integrating both sides of (\ref{num3}) over $(0,x)$, applying the fundamental theorem of calculus and rearranging gives
\begin{equation*}\int_0^x \frac{I_{\nu +n} (t)}{t^\nu}\,\mathrm{d}t = \frac{2(\nu+n+1)}{n+1} \frac{I_{\nu +n+1} (x)}{x^\nu} -\frac{2\nu+n+1}{n+1} \int_0^x \frac{I_{\nu +n +2} (t)}{t^\nu}\,\mathrm{d}t.
\end{equation*}
Applying inequality (\ref{bi2}) to the integral on the right hand-side of the above expression then yields (\ref{bi3}).

(iv) Let $\nu>-1$.   Then integration by parts and inequality (\ref{bi1}) gives
\begin{align*} \int_0^x \mathrm{e}^{-\gamma t} \frac{I_\nu(t)}{t^\nu} \,\mathrm{d}t &= \mathrm{e}^{-\gamma x}\int_0^x \frac{I_\nu(t)}{t^\nu}\,\mathrm{d}t + \gamma \int_0^x \mathrm{e}^{-\gamma t}\bigg(\int_0^t  \frac{I_\nu(u)}{u^\nu} \,\mathrm{d}u\bigg) \,\mathrm{d}t \\
&> \mathrm{e}^{-\gamma x}\int_0^x \frac{I_\nu(t)}{t^\nu}\,\mathrm{d}t + \gamma \int_0^x \mathrm{e}^{-\gamma t} \frac{I_\nu(t)}{t^\nu}  \,\mathrm{d}t-\gamma \int_0^x \frac{\mathrm{e}^{-\gamma t}}{2^\nu\Gamma(\nu+1)} \,\mathrm{d}t,
\end{align*}
whence on evaluating $\int_0^x\mathrm{e}^{-\gamma t} \,\mathrm{d}t=\frac{1}{\gamma}(1-\mathrm{e}^{-\gamma x})$ and rearranging we obtain (\ref{bi4}).

(v) Apply inequality (\ref{bi1}) to inequality (\ref{bi4}).

(vi) We now prove inequality (\ref{bi7}); the assertion that $C_{\nu,n}<2(\nu+n+1)$ is immediate from inequalities (\ref{bi3}) and (\ref{Imon}).  Now, integrating by parts similarly to we did in part (iv), we have
\begin{align*}\int_0^x \mathrm{e}^{-\gamma t} \frac{I_{\nu+n}(t)}{t^\nu} \,\mathrm{d}t &= \mathrm{e}^{-\gamma x}\int_0^x \frac{I_{\nu+n}(t)}{t^\nu}\,\mathrm{d}t + \gamma \int_0^x \mathrm{e}^{-\gamma t}\bigg(\int_0^t  \frac{I_{\nu+n}(u)}{u^\nu} \,\mathrm{d}u\bigg) \,\mathrm{d}t \\
&<\mathrm{e}^{-\gamma x}\int_0^x \frac{I_{\nu+n}(t)}{t^\nu}\,\mathrm{d}t + C_{\nu,n}\gamma \int_0^x \mathrm{e}^{-\gamma t} \frac{I_{\nu+n}(t)}{t^\nu}  \,\mathrm{d}t.
\end{align*}
As we assumed $0<\gamma<\frac{1}{C_{\nu,n}}$, on rearranging we obtain inequality (\ref{bi7}).

(vii) Apply inequality (\ref{bi3}) to inequality (\ref{bi7}).

(viii) Finally, we prove that inequalities (\ref{bi1})--(\ref{bi5}) are tight as $x\rightarrow\infty$ and inequality (\ref{bi3}) is also tight as $x\downarrow0$.  We begin by noting that a straightforward asymptotic analysis using (\ref{Itendinfinity}) gives that, for $0\leq\gamma<1$, $n>-1$ and $\nu\in\mathbb{R}$, 
\begin{equation}\label{eqeq1} \int_0^x \mathrm{e}^{-\gamma t}\frac{I_{\nu+n}(t)}{t^\nu}\,\mathrm{d}t\sim \frac{1}{\sqrt{2\pi}(1-\gamma)}x^{-\nu-1/2}\mathrm{e}^{(1-\gamma)x}, \quad x\rightarrow\infty,
\end{equation}
and we also have
\begin{equation}\label{eqeq2}\mathrm{e}^{-\gamma x}\frac{I_{\nu+n}(x)}{x^\nu}\sim  \frac{1}{\sqrt{2\pi}}x^{-\nu-1/2}\mathrm{e}^{(1-\gamma)x}, \quad x\rightarrow\infty.
\end{equation}
We can use (\ref{eqeq1}) and (\ref{eqeq2}) to prove that inequalities (\ref{bi1})--(\ref{bi5}) are tight as $x\rightarrow\infty$.  For example, for inequality (\ref{bi1}), we have, as $x\rightarrow\infty$,
\begin{align*}\int_0^x \frac{I_\nu(t)}{t^\nu}\,\mathrm{d}t\sim \frac{1}{\sqrt{2\pi}}x^{-\nu-1/2}\mathrm{e}^{x}\quad \text{and} \quad  \frac{I_\nu(x)}{x^\nu}-\frac{1}{2^\nu\Gamma(\nu+1)}\sim \frac{1}{\sqrt{2\pi}}x^{-\nu-1/2}\mathrm{e}^{x},
\end{align*}
and the tightness of the other inequalities is established similarly.

It now remains to prove that inequality (\ref{bi3}) is tight as $x\downarrow0$.  From (\ref{Itend0}), we have on the one hand, as $x\downarrow0$,
\begin{equation*}\int_0^x \frac{I_{\nu+n}(t)}{t^\nu}\,\mathrm{d}t\sim\int_0^x \frac{t^{n}}{2^{\nu+n}\Gamma(\nu+n+1)}\,\mathrm{d}t=\frac{x^{n+1}}{2^{\nu+n}(n+1)\Gamma(\nu+n+1)},
\end{equation*}
and on the other,
\begin{align*}\frac{2(\nu+n+1)}{n+1}\frac{I_{\nu+n+1}(x)}{x^\nu}-\frac{2\nu+n+1}{n+1}\frac{I_{\nu+n+3}(x)}{x^\nu} \sim \frac{x^{n+1}}{2^{\nu+n}(n+1)\Gamma(\nu+n+1)},
\end{align*}
where we used that $u\Gamma(u)=\Gamma(u+1)$.  This proves the claim.
\end{proof}

\begin{remark}When $n=0$ a direct comparison can be made between inequalities (\ref{bi1}) and (\ref{bi2}), which we denote by $l_\nu^{(1)}(x)$ and $l_\nu^{(2)}(x)$, respectively. Applying (\ref{Itend0}), we have that, as $x\downarrow0$,
\begin{equation*}l_\nu^{(1)}(x)\sim\frac{x^2}{2^{\nu+2}\Gamma(\nu+2)} \quad \text{and} \quad l_\nu^{(2)}(x)\sim\frac{x}{2^{\nu+1}\Gamma(\nu+2)},
\end{equation*} 
whilst applying (\ref{Itendinfinity}) gives that, as $x\rightarrow\infty$,
\begin{equation*}l_\nu^{(1)}(x)\sim \frac{\mathrm{e}^{x}}{\sqrt{2\pi x}}\bigg(1-\frac{4\nu^2-1}{8x}\bigg) \quad \text{and} \quad l_\nu^{(2)}(x)\sim\frac{\mathrm{e}^{x}}{\sqrt{2\pi x}}\bigg(1-\frac{4(\nu+1)^2-1}{8x}\bigg).
\end{equation*}
Thus, for all $\nu>-\frac{1}{2}$, $l_\nu^{(2)}(x)$ outperforms $l_\nu^{(1)}(x)$ as $x\downarrow0$, whereas $l_\nu^{(1)}(x)$ outperforms $l_\nu^{(2)}(x)$ as $x\rightarrow\infty$.

Inequality (\ref{bi1}) also serves an important role in the proof of inequality (\ref{bi4}), and consequently inequality (\ref{bi5}).  This is because the bound $l_\nu^{(1)}(x)$ is given in terms of the modified Bessel function $I_\nu(x)$, whereas $l_\nu^{(2)}(x)$ is given in terms of $I_{\nu+1}(x)$. 
\end{remark}

\begin{remark}\label{enveuvj}The constants $C_{\nu,n}$ can be computed numerically.  As an example, we used \emph{Mathematica} to find $C_{0,0}=1.266$, $C_{1,0}=1.682$ $C_{3,0}=2.285$, $C_{5,0}=2.754$ and $C_{10,0}=3.670$.
\end{remark}

\begin{remark}The upper bounds (\ref{bi7}) and (\ref{bi8}) are not tight in the limits $x\downarrow0$ and $x\rightarrow\infty$, but they are of the correct order in both limits ($O(x^{n+1})$ as $x\downarrow0$, and $O(x^{-\nu-1/2}\mathrm{e}^{(1-\gamma)x})$ as $x\rightarrow\infty$). The bounds are simple in that they are given in terms of modified Bessel functions of the first kind and a constant $C_{\nu,n}$ that can bounded above by $2(\nu+n+1)$ or computed numerically as in Remark \ref{enveuvj} if improved accuracy is required.  Whilst for improved accuracy the constant $C_{\nu,n}$ needs to be computed numerically for particular values of $\nu$ and $n$, once this has been done one has a simple bound on the integral $\int_0^x \mathrm{e}^{-\gamma t}\frac{I_{\nu+n}(t)}{t^\nu}\,\mathrm{d}t$, for fixed $\nu$ and $n$, that holds for all $x>0$. However, the bounds are are not entirely satisfactory in that they only hold for $0<\gamma<\frac{1}{C_{\nu,n}}$, whereas one would like the inequalities to be valid for all $0<\gamma<1$.  It should be mentioned that a similar problem was encountered by \cite{gaunt ineq3} in that the upper bounds obtained for $\int_0^x \mathrm{e}^{-\gamma t} t^\nu I_\nu(t)\,\mathrm{d}t$ where only valid for $0<\gamma<\alpha_\nu$, for some $0<\alpha_\nu<1$.  The open problem below, which is analogous to Open Problem 2.10 of \cite{gaunt ineq3}, is considered by this author to be interesting.
\end{remark}


\begin{open}\label{openprob1}Find a constant $M_{\nu,\gamma}>0$ such that, for all $x>0$,
\begin{equation*}\label{bes1} \int_0^x \mathrm{e}^{-\gamma t}\frac{I_\nu(t)}{t^\nu}\,\mathrm{d}t<M_{\nu,\gamma}\mathrm{e}^{-\gamma x}\frac{I_{\nu+1}(x)}{x^\nu}, \quad \nu>-\tfrac{1}{2},\:0<\gamma<1.
\end{equation*}
\end{open}

We end by noting that one can combine the inequalities of Theorems \ref{tiger} and \ref{tiger1} and the integral formulas (\ref{besint6}) and (\ref{besint89}) to obtain lower and upper bounds for a generalized hypergeometric function.  We give an example in the following corollary. 


\begin{corollary}\label{struvebessel}Let $\nu>\frac{1}{2}$. Then, for all $x>0$,
\begin{equation}\label{dob22}I_{\nu}(x)<\frac{x^\nu}{2^{\nu-1}\Gamma(\nu)}{}_1F_2\bigg(\frac{1}{2};\frac{3}{2},\nu;\frac{x^2}{4}\bigg) <2\nu I_\nu(x)-(2\nu-1)I_{\nu+2}(x).
\end{equation}
\end{corollary}

\begin{proof}Combine the integral formula (\ref{besint6}) and inequalities (\ref{bi2}) and (\ref{bi3}) (with $n=0$) of Theorem \ref{tiger1}, and replace $\nu$ by $\nu-1$. 
\end{proof}

\begin{remark}We know from Theorem \ref{tiger} that the two-sided inequality (\ref{dob22}) is tight in the limit $x\rightarrow\infty$, and the upper bound is also tight as $x\downarrow0$.  To elaborate further, we denote by $F_\nu(x)$ the expression involving the generalized hypergeometric function in (\ref{dob22}), and the lower and upper bounds by $L_\nu(x)$ and $U_\nu(x)$.  We now note the bound $\frac{I_{\nu+1}(x)}{I_\nu(x)}>\frac{x}{2(\nu+1)+x}$, $\nu>-1$, which is the simplest lower bound of a sequence of more complicated rational lower bounds given in \cite{nasell2}.  We thus obtain that the relative error in approximating $F_\nu(x)$ by either $L_\nu(x)$ or $U_\nu(x)$ is at most
\begin{align*}\frac{2\nu I_\nu(x)-(2\nu-1)I_{\nu+2}(x)}{I_\nu(x)}-1&=(2\nu-1)\bigg(1-\frac{I_{\nu+2}(x)}{I_{\nu+1}(x)}\frac{I_{\nu+1}(x)}{I_{\nu}(x)}\bigg) \\
&<(2\nu-1)\bigg(1-\frac{x^2}{(2(\nu+2)+x)(2(\nu+1)+x)}\bigg)\\
&=\frac{(2\nu-1)(4(\nu+1)(\nu+2)+(4\nu+6)x)}{(2(\nu+1)+x)(2(\nu+2)+x)},
\end{align*}
which, for fixed $x$, has rate $\nu$ as $\nu\rightarrow\infty$ and, for fixed $\nu$, has rate $x^{-1}$ as $x\rightarrow\infty$.  

We used \emph{Mathematica} to compute the relative error in approximating $F_\nu(x)$ by $L_\nu(x)$ and $U_\nu(x)$, and numerical results are given in Tables \ref{table1} and \ref{table2}.  We observe that, for a given $x$, the relative error in approximating $F_\nu(x)$ by either $L_\nu(x)$ or $U_\nu(x)$ increases as $\nu$ increases.  We also notice from Table \ref{table1} that, for a given $\nu$, the relative error in approximating $F_\nu(x)$ by $L_\nu(x)$ decreases as $x$ increases.  However, from Table \ref{table2} we see that, for a given $\nu$, as $x$ increases the relative error in approximating $F_\nu(x)$ by $U_\nu(x)$ initially increases before decreasing.  This is because the upper bound is tight as $x\downarrow0$. 

\begin{table}[h]
\centering
\caption{\footnotesize{Relative error in approximating $F_\nu(x)$ by $L_\nu(x)$.}}
\label{table1}
{\scriptsize
\begin{tabular}{|c|rrrrrrr|}
\hline
 \backslashbox{$\nu$}{$x$}      &    0.5 &    5 &    10 &    25 &    50 &    100 & 250   \\
 \hline
1 & 0.4948 & 0.2359 & 0.1076 & 0.0409 & 0.0202 & 0.0101 & 0.0040  \\
2.5 & 0.7981 & 0.6245 & 0.3692 & 0.1539 & 0.0784 & 0.0396 & 0.0159  \\
5 & 0.8994 & 0.8321 & 0.6414 & 0.3130 & 0.1678   & 0.0869  & 0.0355 \\
7.5 & 0.9330 & 0.8996 & 0.7822 & 0.4407 & 0.2482 & 0.1318 & 0.0547  \\ 
10 & 0.9498 & 0.9302 & 0.8562  & 0.5426  & 0.3205    & 0.1745 & 0.0735 \\  
  \hline
\end{tabular}
}
\end{table}
\begin{table}[h]
\centering
\caption{\footnotesize{Relative error in approximating $F_\nu(x)$ by $U_\nu(x)$.}}
\label{table2}
{\scriptsize
\begin{tabular}{|c|rrrrrrr|}
\hline
 \backslashbox{$\nu$}{$x$}      &    0.5 &    5 &    10 &    25 &    50 &    100 & 250   \\
 \hline
1 & 0.0051 & 0.2038 & 0.1973 & 0.1034 & 0.0558 & 0.0290 & 0.0118  \\
2.5 & 0.0094 & 0.7325 & 1.1405 & 1.1517 & 0.7143 & 0.3967 & 0.1689  \\
5 & 0.0049 & 0.4995 & 1.5977 & 2.0626 & 1.4411   & 0.8462  & 0.3721 \\
7.5 & 0.0039 & 0.4100 & 1.6473 & 3.4230 & 2.7983 & 1.7750 & 0.8169  \\ 
10 & 0.0032 & 0.3379 & 1.4876  & 4.5026  & 4.2818    & 2.9312 & 1.4119 \\  
  \hline
\end{tabular}
}
\end{table}
\end{remark}

\appendix

\section{An inequality involving the modified Bessel function of the second kind}\label{appa}

In this appendix, we prove the following lemma, which is used in the proof of inequality (\ref{bk1}) of Theorem \ref{tiger}.

\begin{lemma}\label{lemap}Let $n<1$ and suppose $\nu>\frac{5}{2}-2n$.  Then, for $x>0$,
\begin{equation}\label{lemineq}2(\nu+n-2)K_{\nu+n}(x)-(2\nu+n-2)K_{\nu+n-1}(x)+(2-n)K_{\nu+n-3}(x)>0.
\end{equation}
\end{lemma}

We first prove the following elementary inequality. 

\begin{lemma}\label{need3}Let $n<1$ and suppose $\nu>1-\frac{3}{2}n$.  Then, for $x>2(\nu+n)$,
\begin{equation}\label{bk70}\frac{x}{\nu+n-1/2+\sqrt{x^2+(\nu+n-1/2)^2}}>\frac{x-2(\nu+n)}{x-(2-n)}.
\end{equation}
\end{lemma}

\begin{proof}Proving the lemma is equivalent to proving that, for $x>2(\nu+n)$,
\begin{align*}0&<\big(x(x-(2-n))-(\nu+n-\tfrac{1}{2})(x-2(\nu+n))\big)^2\\
&\quad-(x-2(\nu+n))^2(x^2+(\nu+n-\tfrac{1}{2})^2) \\
&=x\big((4n+2\nu-3)x^2-(n^2+2n\nu+n-2\nu-2)x\\
&\quad+(4n^3+8n^2\nu+4n\nu^2-10n^2-18n\nu+4n+4\nu)\big).
\end{align*}
On using the quadratic formula, we see that this inequality holds for 
\begin{align*}x>x^*:=\frac{n^2+2n\nu+n-2\nu-2+\sqrt{(3n+2\nu-2)^2(16\nu+1+14n-7n^2-8n\nu)}}{2(4n+2\nu-3)}.
\end{align*}
We can use a similar argument to the one we just used to show that (\ref{bk70}) holds for $x>x^*$ to show that $x^*<2(\nu+n)$, provided $n<1$ and $\nu>1-\frac{3}{2}n$.  This proves the lemma.
\end{proof}

\noindent{\emph{Proof of Lemma \ref{lemap}.}} Due to identity (\ref{Kidentity}) we can write
\begin{align*}&2(\nu+n-2)K_{\nu+n}(x)-(2\nu+n-2)K_{\nu+n-1}(x)+(2-n)K_{\nu+n-3}(x) \\
&\quad=2(\nu+n-2)\bigg(K_{\nu+n-2}(x)+\frac{2(\nu+n-1)}{x}K_{\nu+n-1}(x)\bigg)\\
&\quad\quad-(2\nu+n-2)K_{\nu+n-1}(x)+(2-n)\bigg(K_{\nu+n-1}(x)-\frac{2(\nu+n-2)}{x}K_{\nu+n-2}(x)\bigg) \\
&\quad=2(\nu+n-2)\bigg\{\bigg(\frac{2(\nu+n-1)}{x}-1\bigg)K_{\nu+n-1}(x)-\bigg(\frac{2-n}{x}-1\bigg)K_{\nu+n-2}(x)\bigg\}.
\end{align*}
Therefore, since $\nu+n-2>0$, proving inequality (\ref{lemineq}) is equivalent to proving that
\begin{equation}\label{need1}\bigg(\frac{2(\nu+n-1)}{x}-1\bigg)K_{\nu+n-1}(x)>\bigg(\frac{2-n}{x}-1\bigg)K_{\nu+n-2}(x),
\end{equation}
for all $x>0$.  This inequality holds for $0<x<2-n$, due to inequality (\ref{cake}) and because the conditions imposed on $n$ and $\nu$ in the statement of the lemma ensure that $\nu>2-\frac{3}{2}n$, and so $2(\nu+n-1)>2-n$.  Due to the fact that $K_{\nu}(x)>0$ for all $x>0$ and $\nu\in\mathbb{R}$, we also immediately see that inequality (\ref{need1}) holds for $2-n\leq x\leq 2(\nu+n-1)$.  It now suffices to prove inequality (\ref{need1}) for $x>2(\nu+n-1)$.  This is equivalent to proving that, for $x>2(\nu+n-1)$,
\begin{equation}\label{need2}\frac{K_{\nu+n-2}(x)}{K_{\nu+n-1}(x)}>\frac{x-2(\nu+n-1)}{x-(2-n)}.
\end{equation}
But Theorem 2 of \cite{segura} (see also \cite{rs16} and \cite{ln10}) states that, for $\nu>\frac{3}{2}-n$ and $x>0$,
\begin{equation*}\frac{K_{\nu+n-2}(x)}{K_{\nu+n-1}(x)}>\frac{x}{\nu+n-3/2+\sqrt{x^2+(\nu+n-3/2)^2}}.
\end{equation*}
Therefore, by Lemma \ref{need3} (with $\nu$ replaced by $\nu-1$), it follows that (\ref{need2}) holds for $x>2(\nu+n-1)$, if $\nu>2-\frac{3}{2}n$ and $n<1$.  (Note that in the proof of this lemma we have required that $\nu+n-2>0$, $\nu>\frac{3}{2}-n$ and $\nu>2-\frac{3}{2}n$, which is guaranteed by our assumption that $n<1$ and $\nu>\frac{5}{2}-2n$.)  This concludes the proof of the lemma.    
\hfill $\Box$

\section{Elementary properties of modified Bessel functions}\label{appb}

Here we list standard properties of modified Bessel functions that are used throughout this paper.  All these formulas can be found in \cite{olver}, except for the inequalities.

The modified Bessel functions $I_{\nu}(x)$ and $K_{\nu}(x)$ are both regular functions of $x\in\mathbb{R}$.  For positive values of $x$ the functions $I_{\nu}(x)$ and $K_{\nu}(x)$ are positive for $\nu>-1$ and all $\nu\in\mathbb{R}$, respectively.  The modified Bessel functions satisfy the following identities and differentiation formulas:
\begin{align}\label{parity}K_{-\nu}(x)&=K_{\nu}(x), \\
\label{Iidentity}I_{\nu +1} (x) &= I_{\nu -1} (x) - \frac{2\nu}{x} I_{\nu} (x), \\
\label{Kidentity}K_{\nu +1} (x) &= K_{\nu -1} (x) + \frac{2\nu}{x} K_{\nu} (x), \\
\label{diffone}\frac{\mathrm{d}}{\mathrm{d}x} \bigg(\frac{I_{\nu} (x)}{x^\nu} \bigg) &= \frac{ I_{\nu +1} (x)}{x^\nu}, \\
\label{diffKi}\frac{\mathrm{d}}{\mathrm{d}x} \bigg(\frac{K_{\nu} (x)}{x^\nu} \bigg) &=- \frac{ K_{\nu +1} (x)}{x^\nu}, 
\end{align}
and have the following asymptotic behaviour:
\begin{align}\label{Itend0}I_{\nu} (x) &\sim \frac{(\frac{1}{2}x)^\nu}{\Gamma(\nu +1)}\bigg(1+\frac{x^2}{4(\nu+1)}\bigg), \quad x \downarrow 0, \: \nu>-1, \\
\label{Itendinfinity}I_{\nu} (x) &\sim \frac{\mathrm{e}^{x}}{\sqrt{2\pi x}}\bigg(1-\frac{4\nu^2-1}{8x}\bigg), \quad x \rightarrow\infty, \: \nu\in\mathbb{R}, \\
\label{Ktend0}K_{\nu} (x) &\sim  2^{\nu -1} \Gamma (\nu) x^{-\nu},  \quad x\downarrow0, \: \nu>0,\\
\label{Ktendinfinity} K_{\nu} (x) &\sim \sqrt{\frac{\pi}{2x}} \mathrm{e}^{-x}\bigg(1+\frac{4\nu^2-1}{8x}\bigg), \quad x \rightarrow \infty, \: \nu\in\mathbb{R}.
\end{align}
Let $x > 0$. Then the following inequalities hold:
\begin{align}\label{Imon}I_{\nu} (x) &< I_{\nu - 1} (x), \quad \nu \geq \tfrac{1}{2},\\
\label{cakkk}K_{\nu} (x) &\leq K_{\nu - 1} (x), \quad \nu \leq \tfrac{1}{2},\\
\label{cake}K_{\nu} (x) &\geq K_{\nu - 1} (x), \quad \nu \geq \tfrac{1}{2}, \\
\label{kineq5} K_{\mu}(x)&<K_\nu(x), \quad 0\leq \mu<\nu.
\end{align}
We have equality in (\ref{cakkk}) and (\ref{cake}) if and only if $\nu=\frac{1}{2}$.  Inequalities (\ref{cakkk}) and (\ref{cake}) for $K_{\nu}(x)$ can be found in \cite{ifantis}.  Inequality (\ref{kineq5}) is immediate from the integral representation $K_\nu(x)=\int_0^\infty\mathrm{e}^{-x\cosh(t)}\cosh(\nu t)\,\mathrm{d}t$, $x>0$, $\nu\in\mathbb{R}$.  The inequality for $I_{\nu}(x)$ can be found in \cite{jones} and \cite{nasell}, which extends a result of \cite{soni}.  A survey of related inequalities for modified Bessel functions is given by  \cite{baricz24}, and refinements of inequalities (\ref{Imon})--(\ref{cake}) can be found in \cite{segura} and references therein.


\subsection*{Acknowledgements}
The author is supported by a Dame Kathleen Ollerenshaw Research Fellowship.


\footnotesize

\begin{thebibliography}{99}
\addcontentsline{toc}{section}{References}

\bibitem{baricz24} Baricz, \'{A}.  Bounds for modified Bessel functions of the first and second kinds. \emph{Proc. Edinb. Math. Soc.} $\mathbf{53}$ (2010), pp. 575--599.

\bibitem{bs09} Baricz, \'{A} and Sun, Y. New bounds for the generalized Marcum $Q$-function. \emph{IEEE Trans. Info. Th.} $\mathbf{55}$ (2009), pp. 3091--3100.

\bibitem{baricz3} Baricz, \'{A}. and Sun, Y.  Bounds for the generalized Marcum $Q$-function. \emph{Appl. Math. Comput.} $\mathbf{217}$ (2010), pp. 2238--2250.

\bibitem{chen} Chen, L. H. Y., Goldstein, L. and Shao, Q.--M.  \emph{Normal Approximation by Stein's Method.} Springer, 2011.

\bibitem{eichelsbacher} Eichelsbacher, P. and Th\"{a}le, C.  Malliavin-Stein method for Variance-Gamma approximation on Wiener space.  	\emph{Electron. J. Probab.} $\mathbf{20}$ no. 123 (2015), pp. 28.


\bibitem{gaunt vg} Gaunt, R. E.  Variance-Gamma approximation via Stein's method.  \emph{Electron. J. Probab.} $\mathbf{19}$ no. 38 (2014), pp. 33.

\bibitem{gaunt ineq1} Gaunt, R. E.  Inequalities for modified Bessel functions and their integrals.  \emph{J. Math. Anal. Appl.} $\mathbf{420}$ (2014), pp. 373--386. 

\bibitem{gaunt ineq3} Gaunt, R. E. Inequalities for integrals of modified Bessel functions and expressions involving them.  \emph{J. Math. Anal. Appl.} $\mathbf{462}$ (2018), pp. 172--190.

\bibitem{gaunt ineq4} Gaunt, R. E. Inequalities for integrals of the modified Struve function of the first kind. \emph{Results Math.} $\mathbf{73}$:65 (2018), pp. 10.

\bibitem{gaunt vg2} Gaunt, R. E. Wasserstein and Kolmogorov error bounds for variance-gamma approximation via Stein's method I. arXiv:1711:07379, 2017.


\bibitem{ifantis} Ifantis, E. K. and Siafarikas, P. D.  Inequalities involving Bessel and modified Bessel functions. \emph{J. Math. Anal. Appl.} $\mathbf{147}$ (1990), pp. 214--227.

\bibitem{jones} Jones, A. L.  An extension of an inequality involving modified Bessel functions.  \emph{J. Math. Phys. Camb.} $\mathbf{47}$ (1968), pp. 220--221.

\bibitem{ln10} Laforgia, A. and Natalini, P. Some Inequalities for Modified Bessel Functions. \emph{J. Inequal. Appl.} (2010), Art. ID 253035, 10 pp.

\bibitem{olver} Olver, F. W. J., Lozier, D. W., Boisvert, R. F. and Clark, C. W.  \emph{NIST Handbook of Mathematical Functions.} Cambridge University Press, 2010.

\bibitem{np12} Nourdin, I. and Peccati, G. Normal approximations with Malliavin calculus: from Stein's method to universality. Vol. 192. Cambridge University Press, 2012.

\bibitem{nasell} N\r{a}sell, I.  Inequalities for Modified Bessel Functions.  \emph{Math. Comput.} $\mathbf{28}$ (1974), pp. 253--256.

\bibitem{nasell2} N\r{a}sell, I. Rational bounds for ratios of modified Bessel functions. \emph{SIAM J. Math. Anal.} $\mathbf{9}$ (1978), pp. 1--11. 

\bibitem{rs16} Ruiz-Antol\'{i}n, D. and Segura. J. A new type of sharp bounds for ratios of modified Bessel functions. \emph{J. Math. Anal. Appl.} $\mathbf{443}$ (2016), pp. 1232--1246.

\bibitem{segura} Segura, J.  Bounds for ratios of modified Bessel functions and associated Tur\'{a}n-type inequalities.  \emph{J. Math. Anal. Appl.} $\mathbf{374}$ (2011), pp. 516--528.

\bibitem{soni} Soni, R. P.  On an inequality for modified Bessel functions. \emph{J. Math. Phys. Camb.} $\mathbf{44}$ (1965), pp. 406--407.

\bibitem{stein} Stein, C.  A bound for the error in the normal approximation to the the distribution of a sum of dependent random variables.  In \emph{Proc. Sixth Berkeley Symp. Math. Statis. Prob.} (1972), vol. 2, Univ. California Press, Berkeley, pp. 583--602.

\end{thebibliography}
\end{document}